\documentclass[11pt]{article}
\usepackage{amsfonts,amsmath}
\usepackage{latexsym}
\usepackage{amsthm}
\usepackage{amscd}
\usepackage{commath}
\usepackage{epsfig}
\usepackage{amssymb}
\usepackage{caption}
\usepackage{enumitem}
\usepackage{mathtools}
\usepackage{booktabs}
\usepackage[table]{xcolor}
\usepackage[toc,page]{appendix}
\usepackage{array}
\newcolumntype{P}[1]{>{\centering\arraybackslash}p{#1}}

\usepackage[english]{babel}

\usepackage[letterpaper,top=2cm,bottom=2cm,left=2.5cm,right=2.5cm,marginparwidth=1.5cm]{geometry}

\usepackage[colorlinks=true, allcolors=blue,bookmarksnumbered]{hyperref}
\hypersetup{
	citecolor = red!80!black
}
\usepackage{color,soul}
\setstcolor{red}
\usepackage{float,subcaption}
\usepackage{graphicx}
\usepackage{stackengine}
\usepackage{tikz}
\usetikzlibrary{intersections}
\usepackage[export]{adjustbox}
\usepackage{array}

\usepackage[capitalise]{cleveref}
\usepackage{microtype}
\usepackage[percent]{overpic}

\usepackage{lineno}
\newtheorem{defn}{Definition}[section]

\newtheorem{theorem}[defn]{Theorem}
\newtheorem{prop}[defn]{Proposition}

\newtheorem{cor}[defn]{Corollary}
\theoremstyle{remark}
\newtheorem{remark}[defn]{Remark}

\numberwithin{equation}{section}
\numberwithin{figure}{section}

\newcommand{\bb}{\begin{equation}}
\newcommand{\ee}{\end{equation}}

\newcommand\HH{\mathbb{H}}

\newcommand{\Hd}{\mathbb H^d}

\newcommand{\Prob}{\mathbb P}

\newcommand{\dist}{\operatorname{dist}}
\newcommand{\dd}{\mathrm d}
\newcommand{\cH}{\mathcal H}
\newcommand{\1}{\mathbf 1}

\DeclareMathOperator{\Jac}{Jac}

\newcommand{\MD}[1]{\textcolor{green!60!black} {[Matteo: #1]}}

 \newcommand{\IPVT}[1]{{\rm IPVT}\big(#1\big)}

 \newcommand{\esp}[1]{\mathbb E \left[ #1 \right]}

\newcommand{\origin}{\mathbf{o}}

\newcolumntype{?}{!{\vrule width 1.5pt }}
\newlength\savedwidth

\newcommand{\tightoverset}[2]{%
  \mathop{#2}\limits^{\vbox to -.5ex{\kern-0.75ex\hbox{$#1$}\vss}}}
  
\newcommand{\R}{\mathbb{R}}
\newcommand{\PP}{\mathbb{P}}
\newcommand{\EE}{\mathbb{E}}

\def\rlabel #1 #2{\begin{equation} \label{#1} #2 \end{equation}}

\def\rproof{\begin{proof}}

\def\Qed{\end{proof}}

\makeatletter
\let\@fnsymbol\@alph
\makeatother

\title{\textsc{Face volume densities of positive-intensity and ideal Poisson--Voronoi tessellations in hyperbolic spaces}}
\author{
Matteo \textsc{D'Achille}\thanks{Institut Élie Cartan de Lorraine, CNRS, Universit\'e de Lorraine,  F-57070, Metz, France.\newline $_{}$\hfill  \href{mailto:matteo.d-achille@univ-lorraine.fr}{\texttt{matteo.d-achille@univ-lorraine.fr}}}
\; \&\; Christoph \textsc{Thäle}\thanks{Fakultät für Mathematik, Ruhr-Universität Bochum, D-44801, Bochum, Germany.\newline $_{}$\hfill  \href{mailto:christoph.thaele@rub.de}{\texttt{christoph.thaele@rub.de}}}
}
\date{}

\begin{document}

\maketitle
\begin{abstract}
\noindent We determine analytically for all $k\in\{0,1,\ldots,d-1\}$ the $k$-volume densities of a Poisson--Voronoi tessellation of intensity $\lambda>0$ in the $d$-dimensional hyperbolic space of constant curvature $-1$. This largely extends previous results of Isokawa in dimensions two and three. As applications, we provide closed form expressions for all face volume densities and all typical face volumes of the ideal Poisson--Voronoi tessellation (IPVT), which is the low-intensity limit as $\lambda\downarrow0$ of the hyperbolic Poisson--Voronoi tessellation. As a main tool we develop a new Blaschke--Petkantschin--type formula in hyperbolic space. \\[0.5cm]
\noindent\textbf{Keywords:} Blaschke--Petkantschin formula; hyperbolic geometry; ideal Poisson--Voronoi tessellation (IPVT); $k$-faces; Poisson--Voronoi tessellation\\
\noindent\textbf{MSC (2020):} 52A55, 60D05, 60G55
\end{abstract}

%{\color{red}
%\begin{itemize}
%\item CT: I changed finite intensity to positive intensity, because this feels much more correct to me. \MD{Fine for me!}
%\item CT: I tried to used consistently the terms $k$-face counting density and $k$-face volume density. Or do you prefer the term intensity for the first quantity? MD: both choices are fine for me, as long as the the fact that these are ${\rm I}_{d,k}$ resp.~${\rm \tilde{I}}_{d,k}$ from $\cite{IPVT}$ is clear (which is the case :) ).
%\item CT: Percolation is not mentioned so far. \MD{Added a couple of sentences at page 3, what do you think?}
%\item MD: Finished Figure 1.1. What do you think? \CT{I like it, but 2 pictures would also be nice. I commented out the middle one. OK?} \MD{OK! I have thus adapted the caption.}
%\item MD: I put the names of major results in textsc. Also, I polished the definition of $A_{d,q}$, and changed o to $\origin$ (to be double-checked).
%\end{itemize}
%}

%\tableofcontents

\section{Introduction and main result}

\iffalse
\paragraph{Poisson--Voronoi percolation in $\HH_d$}

\cite{BS01} \cite{HM24}

\MD{Mention open problem for $\lim_{\lambda \downarrow 0} p_u(\lambda; \HH_3)$ and its connection with Gabioreau's fixed price conjecture. Here or later?}

\MD{Define $p_c(\lambda;\HH_d)$}
\fi

\paragraph{Poisson--Voronoi tessellations and Blaschke--Petkantschin formula in Euclidean spaces.} 

Poisson--Voronoi tessellations are among the most classical models of random
spatial subdivision.  Starting from a stationary Poisson point process in
Euclidean space, one assigns to each point the region consisting of all
locations for which this point is the nearest point of the process.  This
construction produces a random tessellation whose cells describe the domains
of influence of the points of the process.  Due to its simple definition and rich geometric structure, the Poisson--Voronoi tessellation
has become a standard model and has found numerous
applications, for instance in materials science, telecommunications,
geographical modelling and the analysis of spatial data.

Several fundamental mean values of the Euclidean Poisson--Voronoi tessellation
in \(\R^d\) are known analytically.  These include, in particular, the counting
densities of \(k\)-faces and the corresponding \(k\)-volume densities for all
\(k\in\{0,1,\ldots,d-1\}\), see
\cite{MollerPaper89,MollerBook94,SW08} and the references mentioned therein.  Such formulas are basic quantitative
descriptors of the tessellation. They determine how many faces of a given
dimension occur per unit volume and how much \(k\)-dimensional content these
faces carry on average.  They therefore provide a benchmark for simulations,
a point of comparison between different random tessellation models, and a
starting point for the study of more refined characteristics. 

A central role in the derivation of
such formulas is played by Blaschke--Petkantschin--type transformations involving spheres.  These
integral-geometric change-of-variables formulas allow one to separate
location, radius and directional information in configuration integrals, and
thereby turn geometric characteristics of the tessellation into tractable lower-dimensional
integrals.  Beyond Poisson--Voronoi
tessellations, Blaschke--Petkantschin formulas of this type have become versatile tools in
stochastic and computational geometry.  We refer to the survey
\cite{NikiteknoSurvey}, and mention selected applications
in \cite{ChenavierDevillers18,EdelsbrunnerEtAl17,EdelsbrunnerEtAl18,
EdelsbrunnerEtAl19,ReitznerStrotmann}.

\paragraph{Poisson--Voronoi tessellations and Blaschke--Petkantschin formula in hyperbolic spaces.} 
It is natural to ask to what extent the analytic theory for Poisson--Voronoi tessellations survives in
non-Euclidean geometries.  In the present paper we focus on hyperbolic spaces of constant negative curvature $-1$.
Hyperbolic Poisson--Voronoi tessellations retain the basic nearest-neighbour
interpretation of their Euclidean counterparts, but the underlying geometry becomes
substantially different.  In particular, the exponential growth of volume and the absence of a scaling as in Euclidean spaces significantly influence the
topology and geometry of the cells, and make explicit computations far more delicate.  

Compared with the Euclidean case, considerably fewer characteristics of
hyperbolic Poisson--Voronoi tessellations are known in closed form.  Detailed
studies in dimensions two and three were carried out in
\cite{Isokawa2d,Isokawa3d}, where mean values and structural
properties of the corresponding tessellations were investigated based on hyperbolic trigonometry, which is possible only in these low-dimensional cases.
In general
dimension, first explicit progress was made in \cite{BetaStar}.  The approach
there is based on the relation of hyperbolic Poisson--Voronoi tessellations to so-called beta-star polytopes and yields for any dimension $d\geq 2$ and any $k\in\{0,1,\ldots,d-1\}$ an explicit description of the expected number of $k$-faces of the typical cell. These numbers determine in turn the $k$-face counting densities of the
hyperbolic Poisson--Voronoi tessellation. However, these results only concern the combinatorial structure of the
tessellation.  They do not determine the metric size of its lower-dimensional
skeletons.  In particular, no formula for the \(k\)-volume density of the hyperbolic
Poisson--Voronoi tessellation in arbitrary dimension seems to have been available so far. 

The present paper studies the complementary metric quantity, namely the
\(k\)-volume density of the hyperbolic Poisson--Voronoi tessellation, in
arbitrary dimension.  Thus, instead of counting \(k\)-faces, we measure their
total \(k\)-dimensional volume per unit hyperbolic volume.  This provides the
metric counterpart to the face-counting formulas obtained in \cite{BetaStar}.
It also substantially extends the explicit computations in low dimensions
from \cite{Isokawa2d,Isokawa3d} and complements the asymptotic local results
of \cite{CalkaChapronEnriquez} for the face-counting densities of Poisson--Voronoi tessellations on general
Riemannian manifolds. The relevance of such metric
characteristics is already visible in dimension two: the low-intensity limit of
the length density of the one-dimensional skeleton, equivalently the
\(k\)-volume density for \(k=1\), plays a central role in the upper bound for
the Cheeger constant of high genus hyperbolic surfaces obtained in
\cite{BCP25}. This connection will be discussed further below in relation to
the ideal Poisson--Voronoi tessellation.

A key ingredient in our approach is a new Blaschke--Petkantschin--type formula in
hyperbolic space.  It decomposes configuration integrals according to the
geodesic subspace, the centre and the radius of the circumsphere determined by
a collection of points. A boundary case of our formula, corresponding to \(d+1\) points with a unique hyperbolic circumcentre, was previously obtained in \cite[Proposition 2.1.1]{ChapronDiss2018}. The form needed here extends this decomposition to lower-dimensional circumspheres and is therefore suited to the analysis of \(k\)-faces for every $k\in\{0,1,\ldots,d-1\}$.
Since Euclidean Blaschke--Petkantschin formulas of this type have
proved to be a versatile tool in the applications discussed above, it is
natural to expect that a hyperbolic counterpart will be useful beyond the
specific computation carried out here.  We therefore regard this formula as an
object of independent interest. We moreover mention that the ratio of face-counting densities and face-volume
densities available from \cite{Isokawa2d} has been used in~\cite{HM24} to
compute the right derivative at \(\lambda=0\) of the critical probability
\(p_c(\lambda)\) for Poisson--Voronoi percolation on the hyperbolic plane,
where the percolation model is defined by colouring independently the cells of a Poisson--Voronoi
tessellation with intensity \(\lambda\). We leave the investigation of the
implications of the present work for hyperbolic Poisson--Voronoi percolation in dimensions
\(d\geq3\) to future work.

%\MD{Counting faces on the sphere \cite{KabCTVoronoiSphere}}

\begin{figure}
\centering
\includegraphics[width=.47\linewidth]{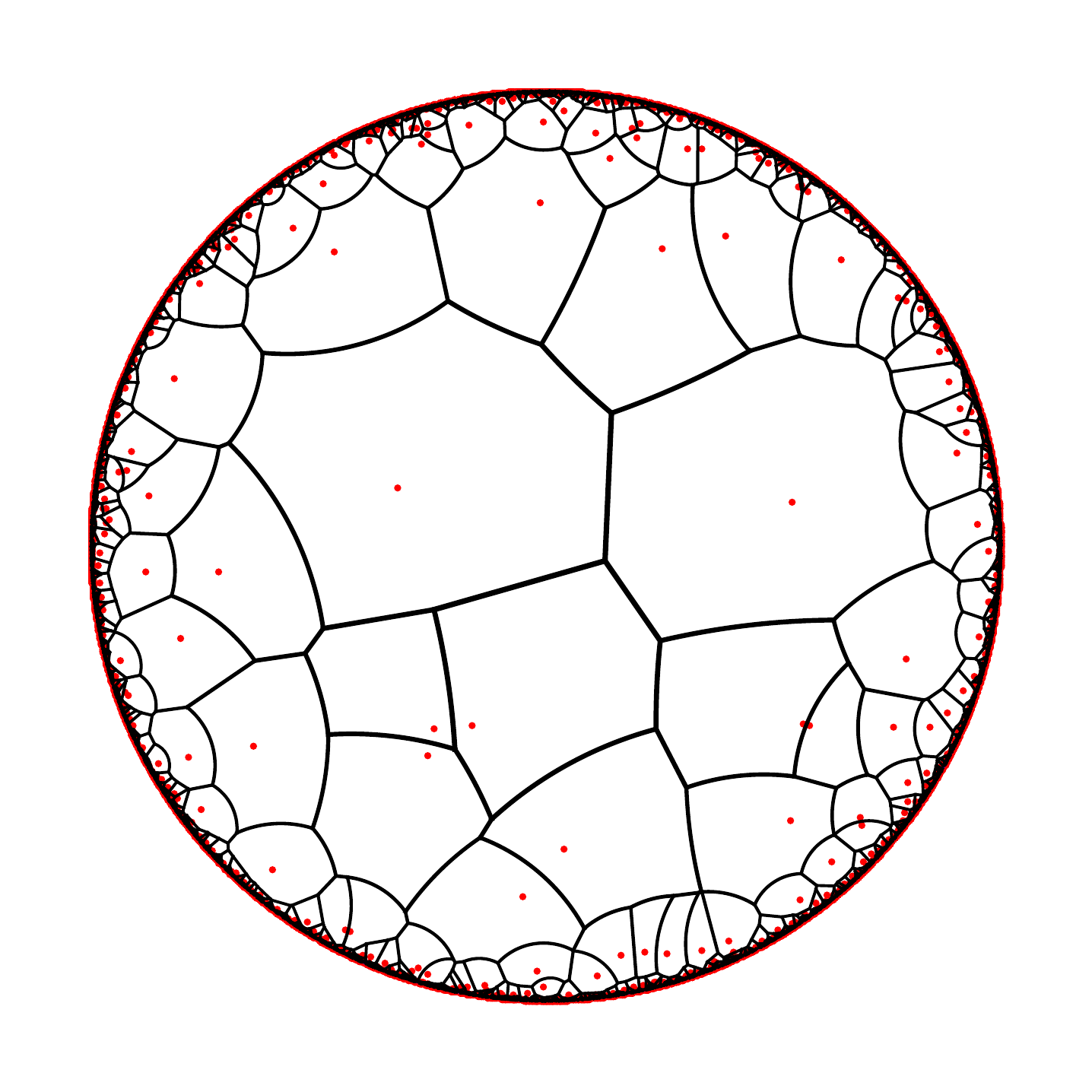} \hspace{5pt}
\includegraphics[width=.47\linewidth]{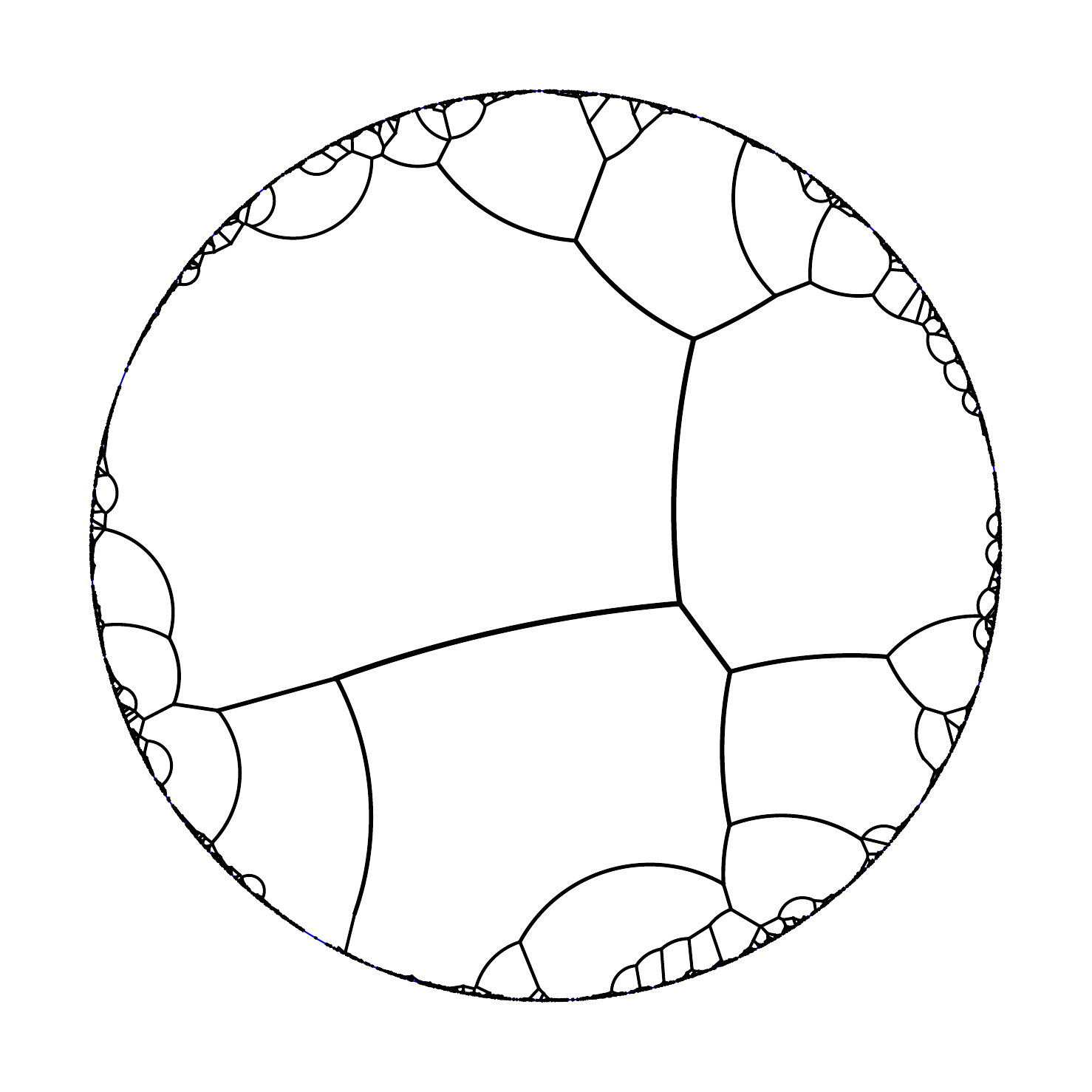} 
\caption{Left: sample of the Poisson--Voronoi tessellation of $\HH^2$ in the conformal disk model at positive intensity (nuclei in red). Right: sample of ${\rm IPVT}(\HH^2)$ coupled via dilation from $\origin$. In both figures, only the first 10000 nuclei have been used (and displayed in the left one).}
\label{fig:pvtipvt}
\end{figure}

\paragraph{Ideal Poisson--Voronoi tessellations.}
The explicit formula for the \(k\)-volume density of a hyperbolic Poisson--Voronoi tessellation we prove makes it possible to
study the low-intensity regime.  This connects the present work to the ideal
Poisson--Voronoi tessellation (IPVT), a recently introduced limiting object obtained
from hyperbolic Poisson--Voronoi tessellations when the number of nuclei per
unit volume tends to zero. The ideal Poisson--Voronoi tessellation is a remarkable new scaling limit
whose cells are unbounded hyperbolic polytopes with a unique ideal point at
infinity. We denote this random tessellation by
\(\IPVT{\HH^d}\), following \cite{IPVT}.

The motivation for ideal Poisson--Voronoi tessellations comes from several
directions.  They were introduced in part in connection with the study of
Cheeger's constant of high genus closed hyperbolic surfaces
\cite{BCP25}, as alredy indicated above, and are also related to questions on invariant random
tessellations and non-amenable geometry, see, for instance,
\cite{bhupatiraju}.  Since their introduction, IPVTs and the geometry of
their cells have found a number of striking applications.  These include connections
with Gaboriau's fixed price conjecture \cite{MiMe23} and estimates for
uniqueness thresholds in Poisson and Bernoulli--Voronoi percolation on product
spaces with non-amenable structure \cite{GR25,AGKRW25}.

For real hyperbolic space \(\HH^d\), \(d\ge2\), the tessellation
\(\IPVT{\HH^d}\) is an isometry-invariant random subdivision of \(\HH^d\)
into countably many unbounded hyperbolic polytopes, each having a unique end, see
Figure~\ref{fig:pvtipvt}, right, for a sample in the case $d=2$.  In the upper half-space model, the zero cell,
that is, the cell containing an arbitrary but fixed point, admits a particularly transparent
description: its law is that of the complement of a union of Euclidean
half-balls whose centres and radii are governed by an explicit Poisson point
process.  This Poissonian representation is one of the central features of
the model, and remains available in related non-Euclidean settings beyond
\(\HH^d\), see \cite{IPVTH2H2}.

As noted above, the face-counting densities of $\IPVT{\HH^d}$ were obtained explicitly in \cite{IPVT}, using the connection with the results of \cite{BetaStar}.
Combining these formulas with the low-intensity limit of the
$k$-volume densities obtained in the present paper yields corresponding
explicit formulas for the mean \(k\)-dimensional volume of the typical
\(k\)-face of the ideal tessellation.  Thus our results provide the missing
metric counterpart to the known face-counting formulas for IPVTs.

\medbreak

%\MD{From now on, statement of the result}

\paragraph{Formulation of the main result.}
Let \(\dist(\cdot,\cdot)\) denote the hyperbolic distance on the $d$-dimensional hyperbollic space \(\Hd\) of constant curvature $-1$, and let \(\cH^m\)
denote \(m\)-dimensional Hausdorff measure, \(m\in\{0,1,\ldots,d\}\), with
respect to this distance.  In particular, \(\cH^d\) is the hyperbolic volume
measure on \(\Hd\).  Fix \(\lambda>0\), and let \(\eta_\lambda\) be a Poisson
point process on \(\Hd\) with intensity measure \(\lambda\cH^d\).  The
Voronoi cell with nucleus \(x\in\eta_\lambda\) is defined by
\[
C(x,\eta_\lambda)
=
\{y\in\Hd: \dist(y,x)\leq \dist(y,z)\text{ for all }z\in\eta_\lambda\}.
\]
The collection of all such cells is the hyperbolic Poisson--Voronoi tessellation with
positive intensity \(\lambda\), see Figure~\ref{fig:pvtipvt}, left, for a sample in the case $d=2$. To simplify comparison with the existing literature we remark that in \cite{IPVT} the intensity was chosen to be $\lambda^{d-1}$. For
\(k\in\{0,1,\ldots,d-1\}\), let \(X_k(\eta_\lambda)\) denote the set of all
\(k\)-faces of this tessellation, that is,
\[
X_k(\eta_\lambda)
=
\bigcup_{x\in\eta_\lambda}
\mathcal F_k\bigl(C(x,\eta_\lambda)\bigr),
\]
where \(\mathcal F_k(P)\) denotes the set of all \(k\)-dimensional faces of
a polytope \(P\).  Fix a bounded Borel set \(W\subset\Hd\) with
\(0<\cH^d(W)<\infty\), and define the \(k\)-face volume density by
\[
D_{d,k}(\lambda)
=
\frac{1}{\cH^d(W)}
\esp{\sum_{F\in X_k(\eta_\lambda)}\cH^k(F\cap W)}.
\]
The expectation on the left-hand side defines an isometry-invariant Radon
measure on \(\Hd\), and hence is a constant multiple of the hyperbolic volume
measure.  Thus \(D_{d,k}(\lambda)\) is independent of the choice of \(W\).
Our first goal is to compute this density explicitly for every positive intensity \(\lambda\).

Throughout this paper we put \(q= d-k\), write
\[
\omega_k={2\pi^{k/2}\over\Gamma(k/2)}
\]
for the area of the \((k-1)\)-dimensional Euclidean unit sphere, and define 
\begin{equation}\label{eq:Adq}
A_{d,q}
\coloneqq
{\omega_{d(q+1)-q}\cdot \omega_{d+1}^q\cdot\omega_{d-q+1}
\over
\omega_{d(q+1)-q+1}}.
\end{equation}
The constant \(A_{d,q}\) has a probabilistic interpretation: $\tfrac{1}{q!\omega_d^{q+1}}A_{d,q}$ is the expected
\(q\)-dimensional Euclidean volume of the random simplex generated by
\(q+1\) independent points sampled uniformly from the Euclidean
\((d-1)\)-dimensional unit sphere, see \cite[Chapter 4]{RandomSimplices}.  Finally, we denote by
\[
b_d(r)=\omega_d\int_0^r\sinh^{d-1}(s)\,\dd s,
\qquad r>0,
\]
the {hyperbolic} volume of a hyperbolic ball of radius \(r\).

We are now prepared to present the main result of this paper. It gives a complete analytic formula for all \(k\)-face volume densities, {$k \in \{0,1,\ldots,d-1\}$, in every dimension $d \geq 2$ and at every positive intensity $\lambda$,} reducing the problem to a {simple} explicit {one-dimensional} radial integral.

\begin{theorem}[\textsc{Face volume densities for Poisson--Voronoi tessellations in \(\Hd\)}]
\label{thm:kVolDensity}
For \(d\geq 2\) and \(k\in\{0,1,\ldots,d-1\}\), we have, with \(q=d-k\),
\[
D_{d,k}(\lambda)
=
{\lambda^{q+1}\over(q+1)!}\,A_{d,q}
\int_0^\infty
\exp\{-\lambda b_d(r)\}
\sinh^{(d-1)(q+1)}(r)\,\dd r .
\]
\end{theorem}

\begin{remark}\label{rem:Laplace}
The integral appearing in Theorem~\ref{thm:kVolDensity} can be viewed as a
Laplace transform.  Define
\[
\psi(s)=\sinh^{d-1}(b_d^{-1}(s)),
\qquad s\geq0 .
\]
Then the change of variables \(s=b_d(r)\) gives
\[
\int_0^\infty
\exp\{-\lambda b_d(r)\}
\sinh^{(d-1)(q+1)}(r)\,\dd r
=
\omega_d^{-1}\,\mathfrak{L}[\psi^q](\lambda),
\]
where \(\mathfrak{L}[\psi^q]\) denotes the Laplace transform of the function
\(\psi^q\). This observation will be used in the proof of Corollary~\ref{cor:IPVTLimit}
below.

The integral also admits a probabilistic interpretation.  If \(\eta_\lambda\) is an isometry-invariant Poisson process in \(\HH^d\) with intensity \(\lambda\), then $\exp\{-\lambda b_d(r)\}=\PP\big(\eta_\lambda(B_\Hd(\origin,r))=0\big)$,
where \(\origin \in\HH^d\) is fixed.  Hence, if \(R_\lambda\) denotes the distance of
\(\origin \) to the nearest point of \(\eta_\lambda\), then
$\PP(R_\lambda>r)=\exp\{-\lambda b_d(r)\}$,
or equivalently \(b_d(R_\lambda)\) has the exponential distribution with
parameter \(\lambda\). Thus the above integral is the same as
\[
\frac{1}{\lambda\omega_d}\,
\EE\big[\sinh^{q(d-1)}(R_\lambda)\big].
\]
Consequently, Theorem~\ref{thm:kVolDensity} may also be written as
\[
D_{d,k}(\lambda)
=
\frac{A_{d,q}}{(q+1)!\,\omega_d}\,
\lambda^q\,
\EE\big[\sinh^{q(d-1)}(R_\lambda)\big].
\]
Thus the \(k\)-face volume density is, up to the constant
\(A_{d,q}/((q+1)!\,\omega_d)\), determined by the \(q\)-th moment of the
hyperbolic surface-growth factor at the nearest-neighbour radius of a typical
point.  This is consistent with the underlying geometric construction: a \(k\)-face has codimension \(q\), and locally arises
from \(q+1\) nuclei on the boundary of an empty hyperbolic ball.  The
void probability of this ball is \(\exp\{-\lambda b_d(r)\}\), while the
factor \(\sinh^{q(d-1)}(r)\) reflects the remaining \(q\) angular degrees of
freedom.
\end{remark}

\section{Applications to the IPVT}

\begin{table}[t]
\centering
\renewcommand{\arraystretch}{1.75}
\setlength{\tabcolsep}{5pt}

\begin{minipage}{0.47\textwidth}
\centering
\[
\begin{array}{c|ccc}
 & d=2 & d=3 & d=4 \\ \hline
k=0
& \dfrac{1}{\pi}
& \dfrac{16\pi}{35}
& \dfrac{1287}{16\pi^2}
\\
k=1
& \dfrac{2}{\pi}
& \dfrac{8\pi}{15}
& \dfrac{16384}{231\pi^2}
\\
k=2
& -
& \dfrac{4}{3}
& \dfrac{105}{8\pi}
\\
k=3
& -
& -
& \dfrac{32}{5\pi}
\end{array}
\]
\caption{Values of $D_{d,k}$ for $2\leq d\leq 4$.}
\label{tab:Ddk}
\end{minipage}
\hfill
\begin{minipage}{0.47\textwidth}
\centering
\[
\begin{array}{c|ccc}
 & d=2 & d=3 & d=4 \\ \hline
k=0
& 1
& 1
& 1
\\
k=1
& \dfrac{4}{3}
& \dfrac{7}{12}
& \dfrac{524288}{1486485}
\\
k=2
& -
& \dfrac{35}{12\pi}
& \dfrac{7\pi}{85}
\\
k=3
& -
& -
& \dfrac{1024\pi}{6075}
\end{array}
\]
\caption{Values of
\(\esp{\cH^k(F^{\rm typ}_{d,k})}\) for \(2\le d\le4\).}
\label{tab:typical-face-volumes}
\end{minipage}
\end{table}

We now turn to the low-intensity regime and its relation with the ideal Poisson--Voronoi tessellation (IPVT).  The latter is obtained as the limit in distribution of hyperbolic Poisson--Voronoi tessellations when the intensity $\lambda$ of the
underlying Poisson process tends to zero and was introduced in \cite{IPVT}. In contrast to the Euclidean case, where changing the intensity merely amounts to a rescaling, the intensity parameter in hyperbolic space has genuine geometric content.  Letting
\(\lambda\downarrow0\) produces a non-trivial limiting random tessellation, denoted by \(\IPVT{\HH^d}\), whose cells are unbounded hyperbolic polytopes with exactly one ideal point at infinity. The explicit formula in Theorem \ref{thm:kVolDensity} allows us to compute the low-intensity limit
$$
D_{d,k} = \lim_{\lambda\downarrow 0}D_{d,k}(\lambda),
$$
which is the $k$-face volume density of \(\IPVT{\HH^d}\) according to \cite[Lemma 4.1]{IPVT}. Values for $d\in\{2,3,4\}$ are given in Table \ref{tab:Ddk}.

\begin{cor}[\textsc{Face volume densities for $\IPVT{\HH^d}$}]\label{cor:IPVTLimit}
For $d\geq 2$ and $k\in\{0,1,\ldots,d-1\}$ we have that, with $q=d-k$,
$$
D_{d,k} = {A_{d,q}\over q+1}{(d-1)^q\over\omega_d^{q+1}} \; .
$$
\end{cor}

\begin{remark}
To ease comparison with the existing literature, we remark that in \cite{IPVT}, the constant which we call $D_{d,k}$ is denoted {by} $\tilde{\rm I}_{d,k}$. In particular, it has been shown in \cite[Proposition 5.3]{IPVT} by computing the slope at the origin of the hole probability for the zero cell of $\IPVT{\Hd}$ that
$$
\tilde{\rm I}_{d,d-1}=\frac{\Gamma({d\over 2})\Gamma(d)}{\Gamma(\frac{d-1}{2})\Gamma(d-\frac{1}{2})}={d-1\over 2}{\omega_{2d-1}\cdot\omega_{d+1}\over\omega_{2d}\cdot\omega_d} \; .
$$
This matches our result for $D_{d,d-1}$.
\end{remark}

\begin{remark}
For \(d=2\), the numerical ordering of the face-volume densities is
\[
D_{2,0}=\frac1\pi<\frac2\pi=D_{2,1}.
\]
For \(d=3\), the sequence of the $k$-volume densities is not monotone, since
\[
D_{3,0}=\frac{16\pi}{35}
<
D_{3,1}=\frac{8\pi}{15}
>
D_{3,2}=\frac43 .
\]
Starting from dimension \(d=4\), however, the face-volume densities are strictly decreasing in the face dimension. In fact, one can show that $D_{d,0}>D_{d,1}>\ldots>D_{d,d-1}$ {for all $d \geq 4$.}
\end{remark}

We record {now} several consequences of Corollary~\ref{cor:IPVTLimit}. The first is concerned with the volume of the typical $k$-face of the IPVT. Using the notation from \cite[Section~4.1]{IPVT}, let \({\rm I}_{d,k}\) denote the \(k\)-face counting density of the IPVT and \(D_{d,k}\) denote the corresponding \(k\)-volume intensity.  While the IPVT does not admit a typical full-dimensional cell in the usual sense, the typical \(k\)-face for $k\in\{0,1,\ldots,d-1\}$ is well defined, since the \(k\)-faces of $\IPVT{\Hd}$ are almost
surely bounded and have positive and finite counting density. We denote by \(F^{\rm typ}_{d,k}\) the typical \(k\)-face of the IPVT, in the sense of the Palm distribution associated with the \(k\)-face process, as defined in \cite[Section~4]{IPVT}. Moreover, putting $c(a)=\pi^{-1/2}\,\Gamma(a)/\Gamma(a-1/2)$, the constant
$$
\mathsf{j}_{d,k} = 2{d\choose k}c\Big({d^2\over 2}\Big)\int_0^\infty \cosh^{-(d^2-1)}u\,\Re\Big({1\over 2}+\mathfrak{i}\,c\Big({d+1\over 2}\Big)\int_0^u\cosh^{d-1}v\,\dd v\Big)^k\dd u
$$
is the angle-sum constant appearing in \cite[Equation~(4.12)]{IPVT}, and 
$$
{\rm IDV}_d={\pi\over(d-1)^{d+1}}{\omega_d^{d+1}\cdot\omega_{d^2-1}\over\omega_{d+1}^d\cdot\omega_{d^2}}
$$
denotes the mean hyperbolic volume of the typical ideal Delaunay simplex. We stress the fact that all these quantities are known analytically from \cite{IPVT,BetaStar,Kabluchko21}. The only missing ingredient for a closed-form expression for
\(\esp{\cH^k(F^{\rm typ}_{d,k})}\) is therefore the \(k\)-volume intensity \(D_{d,k}\), which in turn is provided by Corollary~\ref{cor:IPVTLimit}.  Combining it
with the known formula for the \(k\)-face counting density yields the expression below.  Some explicit values of
\(\esp{\cH^k(F^{\rm typ}_{d,k})}\), for \(d\in\{2,3,4\}\), are displayed in Table~\ref{tab:typical-face-volumes}.

\begin{cor}[\textsc{Mean volumes of typical faces in the IPVT}]\label{cor:TypicalkFace}
For \(k\in\{0,1,\ldots,d-1\}\), with \(q=d-k\),
\[
\esp{\cH^k(F^{\rm typ}_{d,k})}
=
{A_{d,q}\,{\rm IDV}_d\over (d+1)\,\mathsf{j}_{d,k}}\,
{(d-1)^q\over\omega_d^{q+1}} .
\]
\end{cor}

Corollary \ref{cor:IPVTLimit} also has a consequence for ordinary Poisson--Voronoi tessellations of positive intensity in $\Hd$ before passing to the ideal limit.  Although the limiting IPVT does not come with a canonical typical cell in the usual Palm sense,
the Poisson--Voronoi tessellation with intensity \(\lambda>0\) does.  The next corollary shows that the constants \(D_{d,k}\) describe the low-intensity blow-up rate of the expected total \(k\)-volume of the
\(k\)-faces of this typical cell.

\begin{cor}[\textsc{Blowup rates for skeletons of the typical Poisson--Voronoi cell}]\label{cor:LimKFaces}
Let \(Z_{d,\lambda}^{\rm typ}\) denote the typical cell of the Poisson--Voronoi tessellation in \(\mathbb H^d\) with intensity \(\lambda>0\).  Then, for \(k\in\{0,1,\ldots,d-1\}\),
\[
\lim_{\lambda\downarrow0}\lambda\,
\EE\Bigg[
\sum_{F\in\mathcal F_k(Z_{d,\lambda}^{\rm typ})}
\cH^k(F)\Bigg]
=
(d+1-k)D_{d,k}.
\]
\end{cor}

As another application, we mention that the formula in Corollary~\ref{cor:IPVTLimit} also allows us to read off the behaviour of the $k$-volume density $D_{d,k}$ of $\IPVT{\Hd}$ in high dimensions, that is, as $d\to\infty$.  We only record the main regimes.  We use the quotient form of Stirling's formula, which ensures that for fixed
\(a,b\in\mathbb R\),
\[
{\Gamma(x+a)\over\Gamma(x+b)}
=
x^{a-b}
\left(
1+{(a-b)(a+b-1)\over 2x}+O(x^{-2})
\right),
\qquad x\to\infty ,
\]
see \cite{TricomiErdelyi1951}. Now, if the codimension \(q=d-k\) is fixed, then
\[
D_{d,d-q}
=
{d^q\over\sqrt{q+1}}
\left(
1-{q(q+2)^2\over 4(q+1)d}+O(d^{-2})
\right),
\qquad d\to\infty .
\]
For \(q=1\), this gives $D_{d,d-1}={d\over\sqrt 2}
(1-{9\over 8d}+O(d^{-2}))$, in agreement with the discussion after Proposition 5.3 in \cite{IPVT}. If on the other hand \(k\geq0\) is {kept} fixed, then
\[
D_{d,k}
=
{\omega_{k+1}e^{-3/4}\over \sqrt 2\,(2\pi)^{k/2}}\,
d^{\,d-(k+1)/2}e^{-d/2}
\left(
1+{15k-16\over 12d}+O(d^{-2})
\right),
\qquad d\to\infty .
\]
In particular, $D_{d,0}=\sqrt 2\,e^{-3/4}\,d^{d-1/2}e^{-d/2}(1+O(d^{-1}))$ in agreement with \cite[Theorem 4.7]{IPVT}. Finally, if $k=\alpha d+O(1)$ with \(\alpha\in(0,1)\), then Stirling's formula applied on the
logarithmic scale leads to
%Old version
%\[
%\log D_{d,k}
%=
%(1-\alpha)d\log d
%-
%{1-\alpha\over2}d
%-
%{\alpha %d\over2}\log\alpha
%+
%O(\log d).
%\]
%
\[
\log D_{d,k} = (1-\alpha) d \log d -\frac{1}{2}\left(1-\alpha + \alpha \log \alpha \right) d
+
O(\log d) \, .
\]

\section{Notation and preliminaries}\label{sec:Prelim}

In this section we collect some notation and background material that will be used
throughout the paper. The purpose of this section is mainly to provide a common
reference point for the geometric conventions we use.
Some notation will nevertheless be recalled, or introduced in a more specialized
form, at the place where it first becomes relevant. This is meant to keep the
main statements readable while still making the technical parts of the paper
self-contained.

\medbreak

When we work in the Euclidean space $\R^d$ we write $\|\cdot\|$ for the Euclidean norm {and} $\langle\cdot,\cdot\rangle$ for the Euclidean scalar product. The dimension will always be clear from the context. We write \(B_{\mathbb R^d}(z,r)=\{x\in\R^d:\|z-x\|\leq r\}\) and \(B_{\mathbb R^d}^\circ(z,r)=\{x\in\R^d:\|z-x\|<r\}\) for the closed and open Euclidean balls in $\R^d$ with centre \(z\in\R^d\) and radius \(r>0\), respectively. Finally, we write \(\omega_d=2\pi^{d/2}/\Gamma(d/2)\) for the surface measure of the Euclidean unit sphere in \(\mathbb R^d\).

We denote by \(\Hd\) the \(d\)-dimensional {real} hyperbolic space of constant {sectional}
curvature \(-1\), and by \(\dist(\cdot,\cdot)\) its geodesic distance. Its
Riemannian volume measure is denoted by \(\cH^d\). Equivalently, this is the
\(d\)-dimensional Hausdorff measure induced by \(\dist(\cdot,\cdot)\). Throughout
the paper, \(\origin\) denotes an arbitrary but fixed point of \(\Hd\), which we
refer to as the origin. The closed geodesic ball in \(\Hd\) with centre \(x\) and hyperbolic radius \(r>0\) is
denoted by \(B_{\Hd}(x,r)\), and the corresponding open ball by
\(B_{\Hd}^\circ(x,r)\). 

More generally, if \(M\) is a Riemannian manifold, then, by a slight abuse of
notation, \(\cH^s\), \(s\ge0\), denotes the \(s\)-dimensional Hausdorff measure
induced by the Riemannian distance on \(M\). For \(y\in M\), we write \(T_yM\)
for the tangent space at \(y\), endowed with the Riemannian inner product
\(\langle\cdot,\cdot\rangle_y\). Its unit sphere is denoted by
\(S_yM=\{u\in T_yM:\langle u,u\rangle_y=1\}\), and \(\sigma_y\) denotes the
spherical Lebesgue measure on \(S_yM\).

For a Riemannian manifold $M$ the exponential map at \(y\in M\) is denoted by \(\exp_y:T_yM\to M\). Thus, for
\(u\in S_yM\), the curve \(r\mapsto\exp_y(ru)\) is the unit-speed geodesic
starting from \(y\) in {the} direction {of} \(u\), as long as \(ru\) belongs to the domain
of \(\exp_y\). If \(x\in M\) does not lie in the cut locus of \(y\), then \(x\)
has unique geodesic polar coordinates \(x=\exp_y(ru)\), where
\(r>0\) is the geodesic distance of $x$ to $y$ and \(u\in S_yM\). In these coordinates, the Riemannian volume
measure is given by
$\cH^d(\dd x)=\Jac_y(r,u)\,\dd r\,\sigma_y(\dd u)$,
where
\(d=\dim M\) and \(\Jac_y(r,u)\) denotes the radial Jacobian of the exponential
map. In the {case of our interest} \(M=\Hd\), this Jacobian is independent of \(y\) and
\(u\), and equals \(\sinh^{d-1}(r)\). Thus, in hyperbolic space,
$$
\cH^d(\dd x)=\sinh^{d-1}(r)\,\dd r\,\sigma_y(\dd u),\qquad
x=\exp_y(ru).
$$
We refer to \cite{Lee18} for further background material on Riemannian manifolds, especially to \cite[Section 5.3]{Lee18} for the exponential map and to \cite[Corollary 10.17]{Lee18} for the representation of the volume in $\Hd$ in terms of geodesic polar coordinates.

\section{A new Blaschke--Petkantschin--type formula in hyperbolic space}

In the proof of Theorem \ref{thm:kVolDensity} we will make use of a new hyperbolic integral-geometric transformation formula of Blaschke--Petkantschin--type, which is in a similar spirit as the transformation formulas developed in \cite{CalkaChapronEnriquez,ChapronDiss2018}. To present it, we need some further notation. Let \(M\) be an \(m\)-dimensional smooth Riemannian manifold, let \(1\le \ell\le m\), and let \(F:M\to\mathbb R^\ell\) be a smooth map.  For
a point \(z\in M\) put
\[
J_F(z)
=
\sqrt{
\det\bigl(
\langle \nabla F_i(z),\nabla F_j(z)\rangle_z
\bigr)_{i,j=1}^{\ell}
},
\qquad F=(F_1,\ldots,F_\ell).
\]
So, \(J_F\) is the normal Jacobian of \(F\) and $\langle\cdot, \cdot \rangle_z$ denotes the Riemannian inner product in the tangent space $T_zM$ of $M$ at $z$. We use the same notation for product manifolds.

\begin{prop}[\textsc{Hyperbolic Blaschke--Petkantschin formula}]
\label{prop:radial-incidence}
Let \(d\geq2\), let \(k\in\{0,1,\ldots,d-1\}\), and put \(q= d-k\).
For \(\mathbf x=(x_0,\ldots,x_q)\in(\Hd)^{q+1}\), let
$$
L(\mathbf x)=\{y\in\Hd:\dist(y,x_0)=\dist(y,x_1)=\ldots=\dist(y,x_q)\}.
$$
For \(y\in L(\mathbf x)\), write
\(\rho(y,\mathbf x)=\dist(y,x_0)=\dist(y,x_1)=\ldots=\dist(y,x_q)\)
for the common distance.  Then, for every non-negative measurable function
\(h:\Hd\times[0,\infty)\to[0,\infty]\),
\begin{equation}\label{eq:radial-incidence-formula}
\begin{aligned}
&\int_{(\Hd)^{q+1}}
\int_{L(\mathbf x)}
h\bigl(y,\rho(y,\mathbf x)\bigr)\,
\cH^k(\dd y)\,
\cH^{d(q+1)}(\dd\mathbf x)
\\
& =
A_{d,q}
\int_{\Hd}
\int_0^\infty
h(y,r)\,
\sinh^{(d-1)(q+1)}(r)\,\dd r\,
\cH^d(\dd y),
\end{aligned}
\end{equation}
where $A_{d,q}$ is given by \eqref{eq:Adq}.
\end{prop}

%\MD{Comment after Pierre's remark 23.06 in Nancy: The case $k=0$ of the above formula should be in Chapron's thesis (see Proposition 2.1.1. of the file 2018PA1000098.pdf)}

\begin{remark}
\begin{itemize}
\item[(i)] The inner integral of the expression on the left-hand side of \eqref{eq:radial-incidence-formula} is understood in the almost-everywhere sense.  The
configurations \(\mathbf x\) for which \(L(\mathbf x)\) is non-empty but not a
smooth \(k\)-dimensional submanifold form a \(\cH^{d(q+1)}\)-null set and hence
do not affect the value of the integral.
\item[(ii)] The case \(k=0\), equivalently \(q=d\), is already contained in \cite[Proposition 2.1.1]{ChapronDiss2018}. 
Indeed, in this case, for every nondegenerate \((d+1)\)-tuple of points in
\(\Hd\), the set \(L(\mathbf x)\) is either empty or consists of a single point.
In the latter case, this point is the unique hyperbolic circumcentre.
After translating the notation in \cite{ChapronDiss2018} to curvature \(-1\) and integrating out the angular variables, the formula gives precisely the case \(q=d\) of \eqref{eq:radial-incidence-formula}. Proposition~\ref{prop:radial-incidence} shows that the same type of decomposition holds for all \(k\in\{0,\ldots,d-1\}\).
\end{itemize}
\end{remark}

\begin{proof}[Proof of Proposition~\ref{prop:radial-incidence}]
Put \(\mathbb{X}=(\Hd)^{q+1}\).  We consider the set
\(\mathcal I\subset\Hd\times\mathbb{X}\) consisting of all pairs
\((y,\mathbf x)\), with \(\mathbf x=(x_0,\ldots,x_q)\), for which
\(\dist(y,x_0)=\ldots=\dist(y,x_q)>0\).  The part with common radius zero can
only occur when \(x_0=\ldots=x_q=y\).  After projection to \(\mathbb{X}\), this is
contained in the diagonal \(\{x_0=\ldots=x_q\}\), which has
\(\cH^{d(q+1)}\)-measure zero.  It also gives no contribution to the right-hand
side of \eqref{eq:radial-incidence-formula}, since the singleton \(\{0\}\) has
zero Lebesgue measure in the radial variable.

We first describe \(\mathcal I\) as a level set of a suitable function. For that purpose, define 
\[
F:\Hd\times\mathbb{X}\to\mathbb R^q,
\qquad
(y,\mathbf x)\mapsto(F_1(y,\mathbf x),\ldots,F_q(y,\mathbf x)),
\]
where
\[
F_i(y,\mathbf x)
=
\cosh\dist(y,x_i)-\cosh\dist(y,x_0),
\qquad i\in\{1,\ldots,q\}.
\]
The map \((y,x)\mapsto \cosh\dist(y,x)\) is smooth on
\(\Hd\times\Hd\).  Away from the diagonal this follows from the smoothness of
the distance function.  Near the diagonal, one may use geodesic normal
coordinates depending smoothly on the centre \(x\).  In these coordinates the
function \(y\mapsto\cosh\dist(y,x)\) is represented by
$v\mapsto \cosh |v|$,
which is smooth also at \(v=0\), since it has an even power series (this is the
reason why we work with \(\cosh\dist(\cdot,\cdot)\) rather than
\(\dist(\cdot,\cdot)\)).  Hence, \(F\) is smooth on \(\Hd\times\mathbb X\).
Since \(r\mapsto\cosh r\) is strictly increasing on \([0,\infty)\), the equation
\(F(y,\mathbf x)=0\) is equivalent to
\[
\dist(y,x_0)=\dist(y,x_1)=\ldots=\dist(y,x_q).
\]
Thus \(\mathcal I\) is the part of \(F^{-1}(0)\) on which the common distance is
positive.

We claim that \(\mathcal I\) is a smooth submanifold of \(\Hd\times\mathbb{X}\).  Fix
\((y,\mathbf x)\in\mathcal I\), write the common distance as \(r>0\), and
write \(x_i=\exp_y(ru_i)\) with \(u_i\in S_y\Hd\), \(i\in\{0,1,\ldots,q\}\).  Consider
the derivative of \(F\) with respect to the \(\mathbb{X}\)-variables, keeping \(y\)
fixed. By the first variation formula for the Riemannian distance \cite[Theorem 6.3]{Lee18}, $\nabla_x\dist(y,x)$ is the unit vector at $x$ pointing away from $y$ along the geodesic from $y$ to $x$. Hence $\nabla_x\cosh\dist(y,x)=\sinh\dist(y,x)\nabla_x\dist(y,x)$ and therefore, at a point with $\dist(y,x)=r$ we have $\|\nabla_x\cosh\dist(y,x)\|=\sinh r$. Since the \(i\)-th component \(F_i\) depends only on \(x_0\) and \(x_i\), we have
$$
\nabla_{\mathbf x}F_i=\big(-\nabla_{x_0}\cosh\dist(y,x_0),0,\ldots,0,\nabla_{x_i}\cosh\dist(y,x_i),0,\ldots,0\big),
$$
and it follows that $\|\nabla_{\mathbf x}F_i\|^2=\sinh^2r+\sinh^2r=2\sinh^2r$. This gives the diagonal entries of the Gram matrix. For $i\neq j$ the vectors $\nabla_{\mathbf x}F_i$ and $\nabla_{\mathbf x}F_j$ have only one common nonzero component, namely the $x_0$-component. Thus
$$
\langle\nabla_{\mathbf x}F_i,\nabla_{\mathbf x}F_j\rangle_{\mathbf x} = \langle -\nabla_{x_0}\cosh\dist(y,x_0),-\nabla_{x_0}\cosh\dist(y,x_0)\rangle_{x_0} = \sinh^2r,
$$
which yields the off-diagonal entries. Summarizing, we see that the Gram matrix of
\(\nabla_{\mathbf x}F_1,\ldots,\nabla_{\mathbf x}F_q\) is
\(\sinh^2(r)(I_q+\mathbf 1\mathbf 1^\top)\), where ${\bf 1}=(1,\ldots,1)^\top\in\R^q$ is the vector all whose coordinates are equal to one. Since $\det(I_q+\mathbf 1\mathbf 1^\top)=1+\mathbf 1^\top\mathbf 1=1+q$, the determinant of the Gram matrix is
\((q+1)\sinh^{2q}(r)\), and therefore the normal Jacobian of
\(\mathbf x\mapsto F(y,\mathbf x)\) is
\begin{equation}\label{eq:JxF-incidence}
J_{\mathbf x}F(y,\mathbf x)=\sqrt{q+1}\,\sinh^q(r).
\end{equation}
In particular, this is strictly positive as soon as $r>0$, so \(F\) is a submersion along \(\mathcal I\). The implicit
function theorem on manifolds shows that \(\mathcal I\) is a smooth submanifold of dimension
\(d+(q+1)d-q=d(q+1)+k\), see \cite[Chapter 5]{LeeSmoothMF}.

Let \(\pi_\mathbb{X}:\mathcal I\to \mathbb{X}\) and \(\pi_{\Hd}:\mathcal I\to\Hd\) denote the canonical coordinate projections.  We apply the smooth coarea formula for Riemannian manifolds (see \cite[Theorem 13.4.2]{Burago} or (A-2) in the appendix of \cite{Howard1993}) to these two projections.
First, applying it to \(\pi_\mathbb{X}\) gives the left-hand side of
\eqref{eq:radial-incidence-formula} as an integral over \(\mathcal I\), namely
\[
\int_\mathbb{X}
\int_{L(\mathbf x)}
h\bigl(y,\rho(y,\mathbf x)\bigr)\,
\cH^k(\dd y)\,
\cH^{d(q+1)}(\dd\mathbf x)
=
\int_{\mathcal I}
h\bigl(y,\rho(y,\mathbf x)\bigr)\,
J_{\pi_\mathbb{X}}(y,\mathbf x)\,
\cH^{d(q+1)+k}(\dd(y,\mathbf x)),
\]
where we recall that \(J_{\pi_\mathbb{X}}\) denotes the normal Jacobian of the projection \(\pi_\mathbb{X}\).

% Replacement for the coarea-swap paragraph

We now change the order in which the set $\mathcal{I}$ is sliced.  Let
\(J_F\) denote the normal Jacobian of
\(F:\Hd\times\mathbb X\to\mathbb R^q\), and let \(J_yF\) and
\(J_{\mathbf x}F\) denote the normal Jacobians of the partial maps obtained by
keeping \(\mathbf x\), respectively \(y\), fixed.  On the smooth submanifold
\(\mathcal I\subset F^{-1}(0)\), the normal Jacobians of the restricted
coordinate projections satisfy the identities
\[
J_{\pi_\mathbb X}(y,\mathbf x)
=
{J_yF(y,\mathbf x)\over J_F(y,\mathbf x)},
\qquad
J_{\pi_{\Hd}}(y,\mathbf x)
=
{J_{\mathbf x}F(y,\mathbf x)\over J_F(y,\mathbf x)},
\]
compare with the derivation of Equation (1) in \cite{Zaehle1990}. These identities follow by applying the smooth coarea formula in local product coordinates to the regular level-set representation of $\mathcal{I}$.  Consequently,
\[
{J_{\pi_\mathbb X}(y,\mathbf x)\over J_{\pi_{\Hd}}(y,\mathbf x)}
=
{J_yF(y,\mathbf x)\over J_{\mathbf x}F(y,\mathbf x)}.
\]
Applying the smooth coarea formula to
\(\pi_{\Hd}:\mathcal I\to\Hd\) with the integrand
\[
h\bigl(y,\rho(y,\mathbf x)\bigr)
{J_{\pi_\mathbb X}(y,\mathbf x)\over J_{\pi_{\Hd}}(y,\mathbf x)}
\]
therefore gives
\begin{equation}\label{eq:incidence-swap}
\begin{aligned}
&\int_\mathbb{X}
\int_{L(\mathbf x)}
h\bigl(y,\rho(y,\mathbf x)\bigr)\,
\cH^k(\dd y)\,
\cH^{d(q+1)}(\dd\mathbf x)
\\
& =
\int_{\Hd}
\int_{\mathcal I_y}
h\bigl(y,\rho(y,\mathbf x)\bigr)\,
{J_yF(y,\mathbf x)\over J_{\mathbf x}F(y,\mathbf x)}\,
\cH^{d(q+1)-q}(\dd\mathbf x)\,
\cH^d(\dd y),
\end{aligned}
\end{equation}
where $\mathcal I_y=\{\mathbf x\in\mathbb X:(y,\mathbf x)\in\mathcal I\}$ is the fibre above $y$.

We now compute the quotient of Jacobians in \eqref{eq:incidence-swap}.  For
\((y,\mathbf x)\in\mathcal I\), write again \(x_i=\exp_y(ru_i)\) for
\(i\in\{0,1,\ldots,q\}\) with $u_i\in S_y\Hd$ in geodesic polar coordinates around $y$.  As above, by the first variation formula for the Riemannian distance,
\(\nabla_y F_i(y,\mathbf x)=\sinh(r)(u_0-u_i)\) for \(i\in\{1,\ldots,q\}\).  Hence,
\[
J_yF(y,\mathbf x)
=
\sinh^q(r)\,
J(u_0,\ldots,u_q),
\qquad
J(u_0,\ldots,u_q)
\coloneqq
\sqrt{
\det\bigl(
\langle u_i-u_0,u_j-u_0\rangle_y
\bigr)_{i,j=1}^q
}.
\]
Together with \eqref{eq:JxF-incidence}, this yields
\begin{equation}\label{eq:JacobiFraction}
{J_yF(y,\mathbf x)\over J_{\mathbf x}F(y,\mathbf x)}
=
{J(u_0,\ldots,u_q)\over\sqrt{q+1}}.
\end{equation}

It remains to represent the Hausdorff measure $\cH^{d(q+1)-q}$ on \(\mathcal I_y\) in geodesic polar coordinates.  For fixed \(y\), define the map
\[
\Phi_y:(0,\infty)\times(S_y\Hd)^{q+1}\to \mathcal I_y,(r,u_0,\ldots,u_q)\mapsto\bigl(\exp_y(ru_0),\ldots,\exp_y(ru_q)\bigr).
\]
This parametrizes \(\mathcal I_y\)
smoothly and one-to-one, because the exponential map $\exp_y:(0,\infty)\times S_y\Hd\to\Hd\setminus\{y\}$ is a smooth diffeomorphism (there is no cut locus in $\Hd$).  For fixed \(r\), the angular map $u_i\mapsto \exp_y(ru_i)$
from \(S_y\Hd\) onto the geodesic sphere of radius \(r\) has Jacobian
\(\sinh^{d-1}(r)\) for each $i\in\{0,1,\ldots,q\}$.  Thus the \(q+1\) angular variables contribute the factor
\[
\sinh^{(d-1)(q+1)}(r)\,
\sigma_y(\dd u_0)\cdots\sigma_y(\dd u_q).
\]
The radial derivative of \(\Phi_y\) is
\[
\partial_r\Phi_y(r,u_0,\ldots,u_q)
=
\Big({\dd\over\dd r}\exp_y(ru_0),{\dd\over\dd r}\exp_y(ru_1),\ldots,{\dd\over\dd r}\exp_y(ru_q)\Big).
\]
Each component is the velocity vector of a unit-speed geodesic and thus has unit length. Since the
components lie in mutually orthogonal factors of the product
\((\Hd)^{q+1}\), it follows that
$$
\big\|\partial_r\Phi_y\big\|^2 = \sum_{i=0}^q\Big\|{\dd\over\dd r}\exp_y(ru_i)\Big\|^2 = q+1
$$
Moreover, the radial direction is orthogonal to all angular directions in the product tangent space $T_{\bf x}\mathbb{X}$, where ${\bf x}=\Phi_y(r,u_0,\ldots,u_q)$.  Hence
the Jacobian of \(\Phi_y\) is $\sqrt{q+1}\,\sinh^{(d-1)(q+1)}(r)$, and therefore
\[
\cH^{d(q+1)-q}(\dd\mathbf x)
=
\sqrt{q+1}\,
\sinh^{(d-1)(q+1)}(r)\,\dd r\,
\sigma_y(\dd u_0)\cdots\sigma_y(\dd u_q).
\]

Substituting the last identity and \eqref{eq:JacobiFraction} into \eqref{eq:incidence-swap}, the
factors \(\sqrt{q+1}\) cancel, and we obtain
\begin{equation}\label{eq:radial-incidence-directional}
\begin{aligned}
&\int_\mathbb{X}
\int_{L(\mathbf x)}
h\bigl(y,\rho(y,\mathbf x)\bigr)\,
\cH^k(\dd y)\,
\cH^{d(q+1)}(\dd\mathbf x)
\\
& =
\int_{\Hd}
\int_{(S_y\Hd)^{q+1}}
\int_0^\infty
h(y,r)\,
J(u_0,\ldots,u_q)\,
\sinh^{(d-1)(q+1)}(r)\,\dd r\,
\sigma_y(\dd u_0)\cdots\sigma_y(\dd u_q)\,
\cH^d(\dd y).
\end{aligned}
\end{equation}

It remains only to evaluate the $(q+1)$-fold integral over $S_y\Hd$.  After identifying
\(T_y\Hd\) isometrically with \(\mathbb R^d\), we have
\(J(u_0,\ldots,u_q)=q!\Delta_q(u_0,\ldots,u_q)\), where
\(\Delta_q(u_0,\ldots,u_q)\) denotes the Euclidean \(q\)-dimensional volume of
the simplex with vertices \(u_0,\ldots,u_q\).  Hence, by the  spherical
simplex moment formula, see \cite[Theorem~8.2.3]{SW08} and
\cite[Theorem~4.12]{RandomSimplices},
\[
\int_{(S_y\Hd)^{q+1}}
J(u_0,\ldots,u_q)\,
\sigma_y(\dd u_0)\cdots\sigma_y(\dd u_q)
=
A_{d,q}.
\]
The left-hand side is independent of \(y\), by rotational invariance.  Combining
this identity with \eqref{eq:radial-incidence-directional} proves
\eqref{eq:radial-incidence-formula}.
\end{proof}

\section{Proof of Theorem \ref{thm:kVolDensity}}

Fix $d\geq 2$, $k\in\{0,1,\ldots,d-1\}$ and put \(q= d-k\). We write $\eta_{\lambda,\neq}^{q+1}$ for the collection of ordered $(q+1)$-tuples of points of the Poisson point process $\eta_\lambda$. For $\mathbf x=(x_0,\ldots,x_q)\in\eta_{\lambda,\neq}^{q+1}$ write
\begin{equation*}
L(\mathbf x)
=
\{y\in\Hd:\dist(y,x_0)=\dist(y,x_1)=\ldots=\dist(y,x_q)\},
\end{equation*}
and $\rho(y,\mathbf x)=\dist(y,x_0)$ for the common distance.
{We first work on the full-probability event on which the Poisson--Voronoi
tessellation is normal. Thus no point of \(\Hd\) is equidistant from more than
\(d+1\) nuclei, and, for \(q=d-k\), the relative interior of each \(k\)-face is
generated by a unique unordered \((q+1)\)-tuple of nuclei. More precisely, if
\(\mathbf x=(x_0,\ldots,x_q)\) is such a tuple, then the corresponding relative
interior consists of those points \(y\in L(\mathbf x)\) for which $\eta_\lambda\cap B_{\Hd}^{\circ}(y,\rho(y,\mathbf x))=\emptyset$.
The remaining points belong to lower-dimensional faces and therefore do not
contribute to the \(k\)-dimensional Hausdorff measure. Hence, for every bounded
Borel set \(W\subset\Hd\) with $0<\cH^d(W)<\infty$,}
\begin{equation}\label{eq:pathwise-counting}
\sum_{F\in X_k(\eta_\lambda)}\cH^k(F\cap W)
=
\frac{1}{(q+1)!}
\sum_{(x_0,\ldots,x_q)\in\eta_{\lambda,\neq}^{q+1}}
\int_{L(\mathbf x)\cap W}
\1\{\eta_\lambda\cap B_{\Hd}^\circ(y,\rho(y,\mathbf x))=\varnothing\}\,
\cH^k(\dd y),
\end{equation}
where the factor \((q+1)!\) removes the ordering of the same generating $(q+1)$-tuple of nuclei. Taking expectations and applying the multivariate Mecke equation in \cite[Corollary 3.2.3]{SW08} gives
$$
D_{d,k}(\lambda)
=
\frac{\lambda^{q+1}}{(q+1)!\,\cH^d(W)}
\int_{(\Hd)^{q+1}}
\int_{L(\mathbf x)\cap W}
\Prob\{\eta_\lambda\cap B_{\Hd}^\circ(y,\rho(y,\mathbf x))=\varnothing\}\,
\cH^k(\dd y)
\cH^{d(q+1)}(\dd\mathbf x).
$$
Next, we use the Poisson void probability together with the fact that the volume of a hyperbolic ball only depends on its radius. This gives
\begin{align}
D_{d,k}(\lambda)=
\frac{\lambda^{q+1}}{(q+1)!\,\cH^d(W)}
\int_{(\Hd)^{q+1}}
\int_{L(\mathbf x)\cap W}
\exp\{-\lambda b_d(\rho(y,\mathbf x))\}\,
\cH^k(\dd y)
\cH^{d(q+1)}(\dd\mathbf x).\label{eq:mecke}
\end{align}

We now apply Proposition~\ref{prop:radial-incidence} to the right-hand side of
\eqref{eq:mecke}.  Recall that, for
\(\mathbf x=(x_0,\ldots,x_q)\in(\Hd)^{q+1}\), the set \(L(\mathbf x)\) consists
of all points \(y\in\Hd\) having the same distance from
\(x_0,\ldots,x_q\).  
Define the non-negative measurable function
\[
h_\lambda:\Hd\times[0,\infty)\to[0,\infty],(y,r)\mapsto\mathbf 1_W(y)\exp\{-\lambda b_d(r)\}.
\]
Then the integral in \eqref{eq:mecke} can be written as
\[
\int_{(\Hd)^{q+1}}
\int_{L(\mathbf x)}
h_\lambda\bigl(y,\rho(y,\mathbf x)\bigr)\,
\cH^k(\dd y)\,
\cH^{d(q+1)}(\dd\mathbf x).
\]
By Proposition~\ref{prop:radial-incidence}, applied to \(h_\lambda\), we obtain
\[
\begin{aligned}
&\int_{(\Hd)^{q+1}}
\int_{L(\mathbf x)}
h_\lambda\bigl(y,\rho(y,\mathbf x)\bigr)\,
\cH^k(\dd y)\,
\cH^{d(q+1)}(\dd\mathbf x)
\\
& =
A_{d,q}
\int_{\Hd}
\int_0^\infty
h_\lambda(y,r)\,
\sinh^{(d-1)(q+1)}(r)\,\dd r\,
\cH^d(\dd y)
\\
& =
A_{d,q}
\int_{\Hd}
\mathbf 1_W(y)\,\cH^d(\dd y)
\int_0^\infty
\exp\{-\lambda b_d(r)\}
\sinh^{(d-1)(q+1)}(r)\,\dd r
\\
& =
A_{d,q}\,\cH^d(W)
\int_0^\infty
\exp\{-\lambda b_d(r)\}
\sinh^{(d-1)(q+1)}(r)\,\dd r,
\end{aligned}
\]
where we also used Fubini's theorem.
Consequently,
\begin{equation}\label{eq:Ddk-lambda-integral}
D_{d,k}(\lambda)
=
{A_{d,q}\lambda^{q+1}\over (q+1)!}
\int_0^\infty
\exp\{-\lambda b_d(r)\}
\sinh^{(d-1)(q+1)}(r)\,\dd r .
\end{equation}
This completes the proof of Theorem \ref{thm:kVolDensity}.\qed

\section{Proofs of the Corollaries for the IPVT}

\begin{proof}[Proof of Corollary \ref{cor:IPVTLimit}]
We identify the limit of \(D_{d,k}(\lambda)\) as
\(\lambda\downarrow0\).  We use the formula from Theorem \ref{thm:kVolDensity} and put {for ease of notation}
\[
I_\lambda
\coloneqq
\lambda^{q+1}
\int_0^\infty
\exp\{-\lambda b_d(r)\}
\sinh^{(d-1)(q+1)}(r)\,\dd r .
\]
By Remark~\ref{rem:Laplace},
\[
I_\lambda
=
{\lambda^{q+1}\over\omega_d}\,
\mathfrak L[\psi^q](\lambda),
\qquad
\psi(s)=\sinh^{d-1}\bigl(b_d^{-1}(s)\bigr).
\]
Moreover, the large-radius asymptotics for the volume of hyperbolic balls gives
\[
b_d(r)
=
\omega_d\int_0^r\sinh^{d-1}(t)\,\dd t
\sim
{\omega_d\over d-1}\sinh^{d-1}(r),
\qquad r\to\infty,
\]
and hence
\[
\psi(s)^q
\sim
\left({d-1\over\omega_d}\right)^q s^q,
\qquad s\to\infty,
\]
where the notation $\sim$ indicates that the ratio of the left- and the right-hand side tends to $1$. Thus \(\psi^q\) is regularly varying at infinity with index \(q\).
Karamata's Abelian theorem for Laplace transforms
\cite[Chapter XIII.5, Theorem 5]{FellerII} therefore yields
\[
\mathfrak L[\psi^q](\lambda)
\sim
\left({d-1\over\omega_d}\right)^q
\Gamma(q+1)\lambda^{-(q+1)},
\qquad \lambda\downarrow0.
\]
Consequently,
\[
I_\lambda
\longrightarrow
q!\,{(d-1)^q\over\omega_d^{q+1}},
\qquad \lambda\downarrow0 .
\]
Combining this with the representation of \(D_{d,k}(\lambda)\), we
obtain
\[
D_{d,k}
=
\lim_{\lambda\downarrow0}D_{d,k}(\lambda)
=
{A_{d,q}\over(q+1)!}\,
q!\,{(d-1)^q\over\omega_d^{q+1}}
=
{A_{d,q}\over q+1}\,
{(d-1)^q\over\omega_d^{q+1}}.
\]
This completes the proof.
\end{proof}

\begin{proof}[Proof of Corollary \ref{cor:TypicalkFace}]
By \cite[Proposition~4.2]{IPVT},
\[
\esp{\cH^k(F^{\rm typ}_{d,k})}
=
{D_{d,k}\over {\rm I}_{d,k}},
\]
, where we recall that ${\rm I}_{d,k}$ denotes the $k$-face counting density in $\IPVT{\Hd}$.
The statement follows now by combining the result of Corollary \ref{cor:IPVTLimit} with
\[
{\rm I}_{d,k}
=
{d+1\over q+1}\,{\mathsf{j}_{d,k}\over {\rm IDV}_d}
\]
from \cite[Theorem~4.7]{IPVT}.
\end{proof}

\begin{proof}[Proof of Corollary \ref{cor:LimKFaces}]
For a normal isometry-invariant tessellation of $\Hd$, the balance relation between the typical cell and the \(k\)-face volume intensity gives
\[
\EE\Bigg[
\sum_{F\in\mathcal F_k(Z_{d,\lambda}^{\rm typ})}
\cH^k(F)\Bigg]
=
{d+1-k\over \lambda}\,D_{d,k}(\lambda).
\]
Indeed, by normality every \(k\)-face is almost surely incident to exactly \(d+1-k\) full-dimensional cells, and the cell intensity is \(\lambda\), see \cite[Proposition 4.3]{IPVT}.  Multiplying by
\(\lambda\) and using \(D_{d,k}(\lambda)\to D_{d,k}\) as $\lambda\downarrow0$, yields the claim.
\end{proof}

\paragraph{Acknowledgements.} We would like to thank Pierre Calka for bringing the dissertation \cite{ChapronDiss2018} to our attention. M.D'A.~acknowledges support by the ANR project \textit{LOUCCOUM} (ANR-24-CE40-7809). CT was supported by the DFG via SPP 2265 \textit{Random Geometric Systems}. 
This work was initiated during the Mini-Workshop \textit{Hyperbolic Meets Stochastic Geometry} at the Mathematisches Forschungsinstitut Oberwolfach (MFO), see \cite{KT26OWR}. M.D'A.~is grateful to the Department of Mathematics of Ruhr University Bochum for excellent working conditions in the occasion of an invitation during which this work has been partly done. The authors used ChatGPT as a writing and discussion tool in preparing the manuscript. The mathematical content is the responsibility of the authors.

%\nocite{*}

\bibliographystyle{alpha}

\bibliography{biblio}

\end{document}